\documentclass{amsart}
\usepackage[dvips]{color}
\usepackage{graphicx}

\usepackage{latexsym}
\usepackage{amssymb}
\usepackage{amsmath}
\usepackage{amsthm}
\usepackage{amscd}
\theoremstyle{plain}

\newtheorem{thm}{Theorem}
\newtheorem{lem}[thm]{Lemma}
\newtheorem{defin}[thm]{Definition}
\newtheorem{prop}[thm]{Proposition}

\theoremstyle{definition}
\newtheorem{rem}[thm]{Remark}

\newtheorem{ques}{Question}

\newtheorem{cor}[thm]{Corollary}
\begin{document}
\title{Arithmetic progressions in middle $\frac 1 N^{\text{th}}$ Cantor sets}

\author{Jon Chaika}
\email{chaika@math.utah.edu}
\address{Department of Mathematics, University of Utah, 155 S 1400 E Room 233, Salt Lake City, UT 84112}

\maketitle
First to fix some notation. Let $X\subset [0,1]$ be the middle $\frac 1 N^{\text{th}}$ Cantor set. That is $X =\cap_{k=1}^\infty C_k$ where $C_0=[0,1]$ and $C_{k+1}$ is obtained by removing the middle $\frac 1 N^{\text{th}}$ from each connected component of $C_k$.  Notice $C_k$ consists of $2^k$ intervals of size $(\frac{N-1}{2N})^k$. The gaps between these intervals have size at least $\frac{1}{N}(\frac{N-1}{2N})^{k-1}$. 
  Let $a_1,...,a_r$ be numbers and $X+a_r$ be considered modulo 1.  For $\delta>0$ let $X_{\delta}\supset X$ be the set obtained by deleting the middle $N^{th}$ of size at least $\delta$. This is a finite union of intervals.
 \begin{thm}For any $a_1,....,a_{\frac{N}{100\log_2(N)}}$ we have that $\cap_{i=1}^{\frac{N}{100\log_2(N)}} X+a_i \neq \emptyset$.
 \end{thm}
 That is, the middle $\frac 1 N^{th}$ cantor set contains arithmetic progressions and in fact more general configurations of length proportional to $\frac{N}{\log(N)}$.
 
 Broderick, Fishman and Simmons have subsequently proved this statement using variants of Schmidt's game \cite[Theorem 2.1]{BFS}.
 
 \begin{defin} We say an interval $J$ of length $\frac{1}{N^k}$ is k-\emph{good} if 
 $$J  \cap_{i=1}^{\frac{N}{100\log_2(N)}} X_{\frac 1 {N^{k+1}}}+a_i$$ contains $\frac N 2$ disjoint intervals of size $\frac 1 {N^{k+1}}$.
 \end{defin}
 We prove the Theorem by induction using the following Proposition:
 \begin{prop} If $J$ is $k$-good then it contains a subinterval $J'$ which is $k+1$-good.
 \end{prop}
 Notice that by compactness if $J$ is a closed interval and  $$J  \cap_{i=1}^{\frac{N}{100\log_2(N)}} X_{\frac 1 {N^{k+1}}}+a_i\neq \emptyset$$ for all $k$ then $$J  \cap_{i=1}^{\frac{N}{100\log_2(N)}} X+a_i \neq \emptyset.$$
 \begin{lem} Let $L>k$. If $J$ is an interval of size $(\frac{N-1}{2N})^k$ and $I_1,...,I_{2^{L-1}}$ be the intervals removed from $C_{L-1}$ to obtain $C_L$. Then $|\{r:I_r\cap J \neq \emptyset\}|\leq  2^{L-k-1}.$
 \end{lem}
 \begin{proof}
 This is maximized if $J$ is a subinterval of $X_{\frac 1 N(\frac{N-1}{2N})^{k-1}}$. The estimate is achieved for those. To see that it is maximized for subintervals of $X_{\frac 1 N(\frac{N-1}{2N})^{k-1}}$ let us consider a $J$ with $|J|=(\frac{N-1}{2N})^k$ so that the intersections with $I_1,...,I_{2^{L-1}}$ are not contained in one subinterval of $X_{\frac 1 N(\frac{N-1}{2N})^{k-1}}$. So $J$ is contained in $U\cup G\cup V$ where $U$ and $V$ are subintervals of $X_{\frac 1 N(\frac{N-1}{2N})^{k-1}}$ and $G\subset ([0,1]\setminus X_{\frac 1 N(\frac{N-1}{2N})^{k-1}})$ is the gap of size at least $\frac 1 N(\frac{N-1}{2N})^{k-1}$ between them. We assume $U$ is on the left of $V$. First notice no $I_r$ is contained in $G$. Now if $I_r \cap J \cap V \neq \emptyset$ then $J=U+c$ where $c-|G|\geq c- \frac 1 N(\frac{N-1}{2N})^{k-1}\geq d(I_r,q)$ where $q$ is the left endpoint of $V$. Let $p$ be the left endpoint of $U$. There exist $I_L$ with $d(I_L,p)=d(I_r,q)$. Since $|I_L|<|G|$ it follows that $I_L\cap (U+c)=I_L\cap J=\emptyset$. So by sliding U any new intersection with an $I_j$ occurs only after a previous intersection with some $I_r$ has been lost. 
 \end{proof}
 
 \begin{cor} If $J$ is any interval of size $\frac{1}{N^k}$, and $I_1,...I_r$ are the intervals of length exactly $\frac{1}{N^{k}}\delta$ deleted to form $X_{\delta\frac 1 { N^{k}}}$ 
  then
 $$|\{j:I_j\cap J \neq \emptyset\}|\leq 3\cdot 2^{\log_{\frac{2N}{N-1}}\lceil \frac 1 \delta\rceil}.$$
  \end{cor}
  \begin{proof} Let $p=\lceil \log_{\frac{2N}{N-1}}N^k\rceil$. $J$ contains at most parts of 3 subintervals of size $(\frac{N-1}{2N})^p$. Since there are at most $\lceil \frac{1}\delta \rceil$ steps in the inductive process to form $X$ between deleting intervals of size $\frac 1 {N^k}$ and $\delta \frac 1 {N^k}$, The corollary follows by applying the lemma.
  \end{proof}
  \begin{proof}[Proof of Proposition] Consider the subintervals of $J  \cap_{i=1}^{\frac{N}{100\log_2(N)}} X_{\frac 1 {N^{k+1}}}+a_i$ of size $\frac{1}{N^{k+2}}$. By the assumption that $J$ is $k$-good we have at $\frac{N^2}{2}$ disjoint intervals organized into $\frac{N}{2}$ blocks of $N$ consecutive intervals. (We may have other intervals too.) From $X_{\frac{1}{N^{k+1}}}$ to $X_{\frac 1 {N^{k+2}}}$ we can delete portions of at most 
  \begin{multline*}
  3 \log_{\frac{2N}{N-1}}( N) 2^{\log_{\frac{2N}{N-1}}(N)}+3N\log_{\frac{2N}{N-1}} N\leq\\  3 \cdot 2^{(\log_2N)+1}\log_2N+3N\log_2 N\leq 9N\log_2 N
  \end{multline*}
   of them. This estimate follows because $k$ intervals of total measure $c$ can intersect at most $2k+\delta^{-1} c$ disjoint intervals of size $\delta$. There are at most $\log_{\frac{2N}{N-1}}N$ steps, and at each step we remove at most $3  \cdot 2^{\log_{\frac{2N}{N-1}}(N)}$ intervals with total measure at most $\frac 1 {N^k}$. 
  
  We do this for each $X+a_i$ and can delete portions of at most $\frac{N^2}{20}$ intervals of size $\frac 1 {N^{k+2}}$. So by the pigeon hole principle one of the $\frac{N}{2}$ blocks has at least half of its intervals. This is a $k+1$-good subinterval of $J$.
  \end{proof}
  
  \begin{rem} The techniques of this note are a little robust and imply the existence of configurations for bilipshitz images of the middle $\frac 1 N$ cantor set where the bilipshitz constant is not too large depending on $N$. It is natural to ask if there exists $N$ so that the image of the middle $\frac 1 N$ cantor set under any bilipshitz map contains 3 term arithmetic progressions.
  \end{rem}
  \begin{ques}Is the bound found in this note on the order of the correct one? Is it possible to find arithmetic progressions say of order $N$?
  \end{ques}
  \section{Acknowledgments}
  I thank V. Bergelson, L. Fishman and D. Simmons.  In particular D. Simmons saw that the length of AP for the middle $\frac 1 N$ cantor set went to infinity with $N$. This work was supported in part by NSF Grant DMS-1300550.

 \end{document}